\documentclass[10pt]{amsart}

\usepackage{amsmath, amsthm}
\usepackage{eucal}
\usepackage{fullpage}

\usepackage{amssymb}
\usepackage{amscd}
\usepackage{palatino}
\usepackage{latexsym}
\usepackage{epsfig}
\usepackage{graphicx}
\usepackage{amsfonts}
\usepackage{psfrag}
\usepackage{color}
\usepackage{curves}
\usepackage{times}
\usepackage{xcolor}
\usepackage[all]{xy}

\usepackage{epic}

\usepackage{url}
\usepackage{cite}

\usepackage[margin=1in]{geometry}
\usepackage[pdftex]{hyperref}

\newtheorem{thm}{Theorem}[section]

\newtheorem{lem}[thm]{Lemma}

\newtheorem{prop}[thm]{Proposition}
\renewcommand\[{\begin{equation}}
\renewcommand\]{\end{equation}}
\newtheorem{defn}[thm]{Definition}

\theoremstyle{remark}

\theoremstyle{remark}
\newtheorem{rem}[thm]{Remark}
\newtheorem{eg}[thm]{Example}

\numberwithin{equation}{section}

\newcommand{\R}{\mathbb R}
\newcommand{\C}{\mathbb{C}}

\newcommand{\Pro}{\mathbb{P}}

\newcommand{\Z}{\mathbb{Z}}

\newcommand{\mc}{\mathcal}
\newcommand{\mr}{\mathrm}
\newcommand{\mf}{\mathbf}
\newcommand{\mb}{\mathbb}
\newcommand{\mk}{\mathfrak}

\newcommand{\al}{\alpha}
\newcommand{\be}{\beta}
\newcommand{\la}{\lambda}

\newcommand{\Ga}{\Gamma}
\newcommand{\de}{\delta}
\newcommand{\De}{\Delta}
\newcommand{\si}{\sigma}
\newcommand{\Si}{\Sigma}

\newcommand{\ul}{\underline}

\newcommand{\inverse}{^{-1}}

\newcommand{\mkg}{\mk{g}}

\newcommand{\mcL}{\mathcal{L}}

\newcommand{\mcO}{\mathcal{O}}
\newcommand{\mfi}{\mf{i}}
\newcommand{\mfm}{\mf{m}}

\newcommand{\Proj}{\mr{Proj}}
\newcommand{\Spec}{\mr{Spec}}
\newcommand{\SL}{\mr{SL}}

\newcommand{\lee}{\langle}
\newcommand{\ree}{\rangle}

\newcommand*{\boxednumber}[1]{%
      \expandafter\readdigit\the\numexpr#1\relax\relax
}
\newcommand*{\readdigit}[1]{%
      \ifx\relax#1\else
            \boxeddigit{#1}%
             \expandafter\readdigit
      \fi
}
\newcommand*{\boxeddigit}[1]{\fbox{#1}\hspace{-\fboxrule}}

\begin{document}

\title[]{Singular string polytopes and functorial resolutions from Newton-Okounkov bodies}%

\author{Megumi Harada}
\address{Department of Mathematics and
Statistics\\ McMaster University\\ 1280 Main Street West\\ Hamilton, Ontario L8S4K1\\ Canada}
\email{Megumi.Harada@math.mcmaster.ca}
\urladdr{\url{http://www.math.mcmaster.ca/Megumi.Harada/}}

\author{Jihyeon Jessie Yang}
\address{Department of Mathematics\\ Marian University \\ 3200 Cold Spring Road \\ 
Indianapolis, Indiana 46222-1997 \\ U.S.A.}
\email{jyang@marian.edu}


\date{\today}

\begin{abstract}
The main result of this note is that the toric degenerations of flag varieties associated to string polytopes and certain Bott-Samelson resolutions of flag varieties fit into a commutative diagram which gives a resolution of singularities of singular toric varieties corresponding to string polytopes. Our main tool is a result of Anderson which shows that the toric degenerations arising from Newton-Okounkov bodies are functorial in an appropriate sense. We also use results of Fujita which show that Newton-Okounkov bodies of Bott-Samelson varieties with respect to a certain valuation $\nu_{max}$ coincide with generalized string polytopes, as well as previous results by the authors which explicitly describe the Newton-Okounkov bodies of Bott-Samelson varieties with respect to a different valuation $\nu_{min}$ in terms of Grossberg-Karshon twisted cubes. A key step in our argument is that, under a technical condition, these Newton-Okounkov bodies coincide.  
\end{abstract}

\maketitle

\begin{center}
\emph{This is a preliminary version. Comments are welcome.}
\end{center}


\section{Introduction}

The theory of toric varieties yields a powerful dictionary between the algebraic geometry of toric varieties and the combinatorics of fans and polytopes, and it has been of longstanding interest to extend this machinery to other classes of varieties through the study of toric degenerations of a given variety. A rich test case is that of flag varieties. We refer the reader to the survey \cite{FFL}, which explains that two recent developments have invigorated this subject: firstly, the development of the theory of Newton-Okounkov bodies and their associated toric degenerations (as well as relations to equivariant symplectic geometry) \cite{Kaveh-Khovanskii, Kaveh, Anderson}, and secondly, the program of Gross, Hacking, Keel, and Kontsevich concerning cluster varieties \cite{GHK, GHKK}. 

In this note we study a class of toric degenerations of the flag variety $G/B$ associated to string polytopes, constructed using the theory of Newton-Okounkov bodies \cite{Fujita, Anderson}. Part of the motivation for studying a toric degeneration is to extract information about the original variety $X$ from the toric degeneration $X_0$. From this point of view, it would often be more convenient if the degeneration, i.e. the toric variety $X_0$, is non-singular.  The point of departure for us is that the string polytopes -- which include, for example, the famous Gel'fand-Zeitlin polytopes -- are in general singular (i.e. their associated toric varieties are singular). Thus, the question naturally arises whether there is another route by which we may naturally associate non-singular toric varieties to flag varieties.  The theory of Bott-Samelson varieties, which are birational to flag (or Schubert) varieties, and their degenerations to the so-called \emph{Bott towers}, offer plenty of candidates. The relevant question then becomes whether or not the Bott towers can be related in a systematic manner to both the classical string polytopes and the generalized string polytopes introduced in \cite{Fujita}. 
Philosophically, then, we had the following picture in mind, where by ``TV'' we mean ``the toric variety of'': 
\begin{equation}\label{eq: motivation} 
\xymatrix{ 
\textup{ Bott-Samelson varieties } \ar[r]^{\textup{birational}} \ar[d]^{\textup{degenerate}} & \textup{ flag varieties} \ar[d]^{\textup{degenerate}} \\
\textup{ (non-singular) Bott towers } \ar @{.>}[r]^(0.5){?} & \textup{TV(string polytopes)}
}
\end{equation}
and we wondered whether we could complete such a commutative diagram in a natural or systematic way.

The main contribution of this note is Theorem~\ref{thm:main}, which makes~\eqref{eq: motivation} precise. 
We obtain Theorem~\ref{thm:main} by using the observation that the toric degenerations associated to Newton-Okounkov bodies are functorial in a sense explained by Anderson in \cite{Anderson}. We record his result in Theorem~\ref{thm:functoriality}. The fact that we could always construct a diagram~\eqref{eq: motivation} with a \emph{non-singular} toric variety in the lower-left corner required our previous work on Newton-Okounkov bodies of Bott-Samelson varieties in terms of Grossberg-Karshon twisted cubes, and in particular, an observation (recorded in Proposition~\ref{prop:equivalence}) that, under a technical hypothesis, the Newton-Okounkov bodies computed in \cite{HY1} coincide with those computed by Fujita in \cite{Fujita}. From there, a technical argument (Proposition~\ref{prop: nice m}) based on principles similar to those in \cite{HY1} allows us to construct suitable non-singular Bott towers to place in the lower-left corner of~\eqref{eq: motivation}. 
In order to see that the toric degenerations obtained through Anderson's construction \cite{Anderson} correspond to the expected (normal) toric varieties, we use results of \cite{HY1} that certain polytopes are lattice polytopes.

We end with brief remarks about our motivation for this paper and potential directions for future work. First, in this note we restrict to the case of $G/B$ for simplicity, but it should be possible to generalize our work to Schubert varieties in $G/B$. Second, we were partly motivated by the work of Kiritchenko, Smirnov, and Timorin \cite{KST} which gives a polytope-theoretic formulation of Schubert calculus using the Gel'fand-Zeitlin polytopes. As already hinted above, an issue that arises in \cite{KST} is that the Gel'fand-Zeitlin polytopes are singular, and in their work they consider associated non-singular polytopes, constructed from the Gel'fand-Zeitlin polytopes. Thus, we wondered whether these non-singular polytopes have a geometric meaning. Finally, recent work of e.g. Abe, Horiguchi, Murai, Masuda, Sato \cite{AHMMS} and others \cite{ADGH, AHHM} suggest that there are interesting relationships between: the cohomology rings of Hessenberg varieties (which are certain subvarieties of the flag variety) and their associated volume polynomials, on the one hand, and string polytopes and (the volumes of unions of) their faces, on the other hand. Moreover, as in the Schubert calculus considerations of e.g. \cite{KST}, the singularity of the string polytopes arises as an issue. Thus, one of our motivations was that our resolutions may potentially be used in the theory of Hessenberg varieties and their volume polynomials.

\medskip
\noindent \textbf{Acknowledgements.} We thank Dave Anderson for a valuable conversation which gave us the idea for this manuscript.

\section{Background: Newton-Okounkov bodies and toric degenerations}\label{sec:background}

In this section, we give a precise statement (Theorem~\ref{thm:functoriality}) of Anderson's results concerning the functoriality of the toric degenerations associated to Newton-Okounkov bodies. This work of Anderson is both the motivation and the main tool used in this paper. In order to state Theorem~\ref{thm:functoriality} we need to review the theory of Newton-Okounkov bodies and their associated degenerations. Since this material is treated elsewhere we keep discussion very brief. For details we refer to e.g. \cite{Kaveh-Khovanskii, Kaveh}.

Let $X$ be a projective algebraic variety over $\C$. The theory of Newton-Okounkov bodies associates to $X$, together with some auxiliary data, a convex body (i.e., a compact convex set) in a real vector space, called the Newton-Okounkov body of $X$. 
Throughout, we fix $< = <_{\mr{lex}}$ to be the usual lexicographic order on $\Z^n$, i.e., 
$(x_1,\dots,x_n)<_{\mr{lex}}(y_1,\dots,y_n)$ if and only if the leftmost non-zero value in the sequence, $x_1-y_1, x_2-y_2,\dots, x_n-y_n$,  is strictly negative.

\begin{defn}\label{def:prevaluation}
  Let $V$ be a vector space over $\C$. A {\bf  ($\Z^n$-)prevaluation} on $V$   is a function  
  \[v:V\setminus\{0\}\rightarrow \Z^n\] such that:
        \begin{enumerate}
          \item $v(\la f)=v(f)$ for all $0\ne f\in V$ and $0\ne\la\in \mf{k}$, and 
          \item $v(f+g)\ge \min\{v(f), v(g)\}$ for all $f, g\in V$ with $f, g, f+g$ all nonzero.          \end{enumerate}
    \end{defn}

It is not hard to see that any set of non-zero vectors in $V$ with distinct values under a prevaluation $v$ is linearly independent \cite[Proposition 2.3]{Kaveh-Khovanskii}. From this it also follows that $\dim(V) \ge \lvert v(V \setminus \{0\}) \rvert$ if $V$ is finite-dimensional.

We can generalize the notion of prevaluations on vector spaces to the notion of valuations on algebras. Note that our choice of total order has the property that it respects the addition, i.e., $a<b$ implies $a+c< b+c$ for any $a,b,c\in \Z^n$. 

 \begin{defn} Let $R$ be an algebra over $\C$ and let $<$ be a total order on $\Z^n$ respecting addition. A {\bf $\Z^n$- valuation} on $R$   is a prevaluation $v:R\setminus\{0\}\rightarrow\Z^n$ which additionally satisfies the following:
             \[v(fg)=v(f)+v(g)   \textup{ for any } f, g\in R\setminus\{0\}. \]
\end{defn}

One of the main motivations to associate convex bodies to an algebraic variety $X$ is to study the variety from a combinatorial point of view. For this purpose, it turns out to be useful to impose an extra condition on the (pre)valuation. Let $v:V\setminus\{0\}\rightarrow \Z^n$ be a prevaluation. For any $\al\in \Z^n$ we define a subspace $V_{\ge \al}$ of $V$ by 
 \[V_{\ge\al}:=\{f\in V: \text{ either } v(f)\ge \al \text{ or } f=0\}.\] 
 By the condition (2) in Definition~\ref{def:prevaluation}, the union $\bigcup_{\beta > \al} V_{\ge \beta}$ is a subspace of $V_{\ge \al}$. We then call the quotient space  
  \[\hat{V}_{\al}:=V_{\ge \al}/\bigcup_{\be>\al} V_{\ge \be}\]
the \textbf{leaf above $\al$}.

\begin{defn}\label{def:one-dim leaves}
A prevaluation $v:V\setminus\{0\}\rightarrow \Z^n$ is said to have {\bf one-dimensional leaves} if 
$\mathrm{dim}(\hat{V}_{\al}) \leq 1$ for all $\al$.
\end{defn}

From the definitions it is straightforward to see that 
 \[\dim V=\sum_{\al\in\Z^n} \dim \hat{V}_{\al}\]
assuming that $V$ is finite-dimensional, 
and it then follows from Definition~\ref{def:one-dim leaves} that 
if $v$ has one-dimensional leaves, then the inequality $\dim(V) \ge \lvert v(V \setminus \{0\}) \rvert$ is in fact an equality, i.e., 
\[
\dim(V) = \lvert v(V \setminus \{0\}) \rvert.
\]

Now we recall the construction of a Newton-Okounkov body of a complex projective variety $X$ associated to certain choices, as detailed below.

\begin{itemize}
 \item [Step 1.] Fix an embedding of $X$ in a projective space $\Pro^N$.
 \item [Step 2.] Let $R=\oplus_{d\in \Z_{\ge 0}}R_d$ denote the homogeneous coordinate ring of $X$ in $\Pro^N$. Here $R_1=H^0(X,\mc{O}_X(1))$ and $R_0=\C$, where $\mcO_X(1)$ is the tautological line bundle $\mcO(1)$ on $\Pro^N$ restricted to $X$. (Note each $R_k$ is finite-dimensional.) 

\item [Step 3.] Let $v:R\setminus\{0\}\rightarrow \Z^n$ be a fixed choice of a ($\Z^n$-)valuation on $R \setminus \{0\}$ with one-dimensional leaves on each $R_k \setminus \{0\}$. 
\item [Step 4.] Define an extention $\hat{v}$ of $v$ to account for the $\Z_{\geq 0}$-grading on $R$ as follows: 
 \[\hat{v}:R\setminus\{0\}\rightarrow \Z_{\ge 0}\times \Z^n, \quad \hat{v}(f):=(d, v(f_d)),\] where $f_d$ is the homogeneous component of $f$ of degree $d :=\deg(f)$.  Here we assign a total order on $\Z_{\ge 0}\times \Z^n$ by 
 \begin{equation}\label{eq: extended v order}
 (d_1, \ul{x})< (d_2, \ul{y}) \text{ if and only if either } d_1<d_2, \text{ or } d_1=d_2 \text{ and }-\ul{x}<_{\mr{lex}}- \ul{y} \quad(\text{equivalently, } \ul{y}<_{\mr{lex}}\ul{x}),
 \end{equation}
  where $d_1,d_2\in\Z_{\ge 0}$ and $ \ul{x}, \ul{y}\in\Z^n.$

\item [Step 5.] Let $\Ga:=\Ga(R,v)\subset \Z_{\ge 0}\times \Z^n$ denote the image of $\hat{v}$. (Note that $\Ga$ is a $\Z_{\ge 0}$-graded semigroup.)
\item [Step 6.] Let $\mr{Cone}(R,v)$ denote the closure of the convex hull of $\Ga$ in the real vector space $\R\times\R^n$.

\item [Step 7.] Define {\bf the Newton-Okounkov body of $(R,v)$}  as the ``level-$1$ slice'' of $\mr{Cone}(R,v)$, i.e., \[\De(R,v):=\mr{Cone}(R,v)\cap (\{1\}\times \R^n).\]
\end{itemize}

We now turn our attention to degenerations of $X$ associated to the above construction. For details we refer to \cite{Anderson}. 
Our choice of $\Z^n$-valuation $v$ on $R$ in Step 3 above gives rise to a filtration of $R$ indexed by the semigroup $\Ga$ as follows: Let 
 \[R_{\le (d, \ul{x})}:= \{f\in R \hspace{2mm} | \hspace{2mm}  \hat{v} (f)\le (d, \ul{x})\}\] 
  It is immediate from the definitions that if $(d, \ul{x})\le (d',\ul{x}')$ then $R_{\le (d,\ul{x})}\subset R_{\le (d',\ul{x}')}$. Thus $R$ is $\Ga$-filtered. Therefore, there is a natural associated $\Ga$-graded algebra
  \[\mr{gr} R:= \mr{gr}_v R:=\bigoplus_{(d,\ul{x})\in \Ga}R_{\le (d,\ul{x})}/  R_{< (d,\ul{x})}.\]

\begin{prop}[\protect{\cite[Proposition 3]{Anderson}}]\label{prop: Anderson}
  Suppose that the associated graded algebra $\mr{gr}R$ is finitely generated. Then there is a finitely generated, $\Z_{\ge 0}$-graded, flat $k[t]$-subalgebra $\mc{R}\subset R[t]$ such that
  \begin{itemize}
    \item $\mc{R}/t\mc{R}\cong \mr{gr} R$ and
    \item $\mc{R}[t\inverse]\cong R[t,t\inverse]$ as $k[t,t\inverse]$-algebras
  \end{itemize}
Geometrically, the above implies that there is a flat family of affine varieties $\hat{\mk{X}}\rightarrow \mb{A}^1$ such that the general fiber $\hat{X}_t$ is isomorphic to $\hat{X}=\Spec R$  for $t\ne 0$, and the special fiber $\hat{X}_0$ has an action of the torus $\C^*\times (\C^*)^n$.
Furthermore, there is another $\Z_{\ge 0}$-grading on $\mc{R}$ which is compatible with  the $\Z_{\ge 0}$-grading of $R$. Taking $\Proj$ with respect to this grading, we obtain a projective flat family $\mk{X}=\Proj \mc{R}\rightarrow\mb{A}^1$, with general fiber isomorphic to $X=\Proj R$ and special fiber $X_0=\Proj (\mr{gr}R)$ equipped with an action of $(\C^*)^n$.
\end{prop}

We observed above that if $v$ has one-dimensional leaves then there is a relation between the image of $v$ and a basis for $V$. The following lemma records another consequence of the property that $v$ has one-dimensional leaves, which is essential in constructing toric degenerations associated to Newton-Okounkov bodies. 

\begin{prop}[\protect{\cite[proof of Theorem 2]{Anderson}}]\label{prop:1 dimensional leaves}
  Suppose that the valuation $v$ on $R$ has one-dimensional leaves. Then  the associated graded algebra $\mr{gr} R$ is isomorphic to the semigroup algebra $\C[\Ga]$.
\end{prop}

We can now make precise the statements about toric degenerations associated to Newton-Okounkov bodies. 

\begin{thm}[\protect{\cite[Theorem 1]{Anderson}}]\label{thm:NOBY degeneration}
  Suppose that the associated graded algebra $\mr{gr} R$ is finitely generated and the valuation $v$ on $R$ has one-dimensional leaves. Then $X=\Proj R$ admits a flat degeneration to the (not necessarily normal) toric variety $X_0=\Proj k[\Ga]$.  That is, there is a flat family $\mk{X}\rightarrow \mb{A}^1$  with general fiber isomorphic to $X=\Proj R$ and special fiber $X_0=\Proj \hspace{1mm} \C[\Ga]$.  Moreover, the normalization of $X_0$ is the (normal) toric variety corresponding to the Newton-Okounkov polytope $\De(R,v)$.

\end{thm}

Furthermore,  it is shown in \cite{Anderson} these toric degenerations are ``functorial'' in a certain sense which we now make precise. This point of view was the motivation for the current note.  To state the relevant results, we need to introduce a notion of {\em compatibility}. Suppose we have the following.

\begin{itemize}
  \item
Let $R$ and $R'$ be $\Z_{\ge 0}$-graded algebras over $\C$ and let $\phi:R\rightarrow R'$ be a graded ring homomorphism.

\item Let $v$ (resp. $v'$) be $\Z^n$- (resp. $\Z^{n'}$-) valuations on $R$ (resp. $R'$).  Also let $h:\Z^n\rightarrow \Z^{n'}$ be a map of ordered groups such that the induced map $\mr{id}\times h:\Z\times\Z^n\rightarrow\Z\times\Z^{n'}$ is a map of ordered groups with respect to the order defined in~\eqref{eq: extended v order}. 
\end{itemize}

In the setting above, we make the following definition \cite[Definition 4]{Anderson}. 

\begin{defn}
We say that the valuations $v$ and $v'$ are {\bf compatible with $\phi$ and $h$} if the following diagram

\[\xymatrix{R\setminus J\ar[rr]^{\phi}
\ar[d]^v && R'\setminus\{0\}\ar[d]^{v'}\\
\Z^n\ar[rr]^{h}&& \Z^{n'}
}\]

commutes, where $J := \mathrm{ker}(\phi)$ is the kernel of the given map $\phi$.
\end{defn}

The following is a summary of results obtained in \cite{Anderson}.

\begin{thm}[\protect{\cite[cf. Proposition 4, Example 8]{Anderson}}]\label{thm:functoriality}

We use notation as above. 
\begin{enumerate} 
  \item (Algebraic statement) Suppose that the valuations $v$ and $v'$ are compatible with $\phi$ and $h$. Suppose that $\mr{gr} R$ and $\mr{gr} R'$ are finitely generated (hence $R$ and $R'$ satisfy the hypothesis in Proposition~\ref{prop: Anderson}. Then there exist Rees algebras $\mc{R}$ and $\mc{R}'$ associated to $R$ and $R'$ as described in Proposition~\ref{prop: Anderson} and a $k[t]$-algebra homomorphism $\Phi:\mc{R}\rightarrow \mc{R'}$, which preserves the $\Z_{\ge 0}\times \Z_{\ge 0}$-gradings.

  \item (Geometric statement) Let  $\mcL$ be a very ample line bundle on a projective algebraic variety $X$. Let $V=H^0(X,\mcL)$ and $R=R(V)$ be the homogeneous coordinate ring of $X$ in $\Pro(V^*)$ so that $X=\Proj R$.  Suppose $v$ is a valuation on $R$ such that the semigroup $\Ga(R, v)$ is finitely generated. Let $V'\subset V$ be a subspace such that the corresponding rational map $\psi: X=\Proj(R)\rightarrow X'=\Proj R'$ is a birational isomorphism, where $R' = \oplus_{m\ge 0} V'^m$ is the graded $\C$-algebra generated by $V'$. Also suppose that the semigroup $\Ga(R', v)$ is finitely generated. Then $\De(R',v)\subset \De(R, v)$ and the birational morphism $\psi:X\rightarrow X'$ degenerates to a birational morphism of toric varieties $\psi_o: \Proj \hspace{1mm} \C[\Ga(R, v)] \rightarrow \Proj \hspace{1mm} \C[\Ga(R', v)]$. That is, there is a commutative diagram
      \[\xymatrix{\mk{X}(V)\ar[rr]^{\Psi}
\ar[dr] && \mk{X}(V')\ar[dl]\\
&\C&
}\] such that $\Psi_t\cong \psi$ for $t\ne 0$, and $\Psi_0=\psi_0$.

\end{enumerate}

\end{thm}

\section{Newton-Okounkov bodies of Bott-Samelson varieties}

In this section, we review and compare two constructions of Newton-Okounkov bodies associated to Bott-Samelson varieties \cite{Fujita, HY2}. The main result of this section is that, under a technical condition and up to a simple linear transformation, these Newton-Okounkov bodies actually coincide (Proposition~\ref{prop:equivalence}). 

We begin with a brief discussion of Bott-Samelson varieties and their line bundles. Let $G$ be a connected and simply connected complex semisimple algebraic group and let $\mathfrak{g}$ denote its Lie algebra. Let $H$ be a Cartan subgroup of $G$, and $B$ a Borel subgroup of $G$ with $H \subset B \subset G$. We denote the rank of $G$ by $r$. 
Let $\Lambda$ denote the weight lattice of $G$
and $\Lambda_\R = \Lambda \otimes_\Z \R$ its real form. The Killing form\footnote{The Killing form is naturally defined on the Lie algebra of $G$ but its restriction to the Lie algebra $\mathfrak{h}$ of $H$ is positive-definite, so we may identify $\mathfrak{h} \cong \mathfrak{h}^*$.}
on $\Lambda_\R$ is denoted by
 $\langle \cdot, \cdot \rangle$.
For a weight $\alpha \in \Lambda$, we let $e^\alpha$ denote the corresponding multiplicative character $e^\alpha: B \to \C^*$. We let $\{\alpha_1, \ldots, \alpha_r\}$ be the ordered set of positive simple roots with respect to the choices $H \subset B \subset G$ and denote by $\alpha_i^\vee$ the 
corresponding coroots. Recall that a root $\alpha$ and its corresponding coroot $\alpha^\vee$ satisfies
\[
\alpha^\vee := \frac{2 \alpha}{\langle \alpha,\alpha \rangle}
\]
so in particular, $\langle \alpha,
\alpha^\vee \rangle = 2$ for any simple root $\alpha$.
Furthermore, for a simple root $\alpha$
let $s_\alpha:\Lambda\rightarrow \Lambda,
\lambda\mapsto\lambda-\langle\lambda,\al^{\vee}\rangle \al,$ denote the associated simple reflection; these generate the Weyl group $W$.
We let $\{\varpi_1, \ldots, \varpi_r\}$ denote the set of fundamental weights satisfying $\langle \varpi_i, \alpha^\vee_j \rangle = \delta_{i,j}$.
Finally, for
 a simple root $\alpha$, we write $P_\alpha := B \cup
  B s_\alpha B$ for the
  minimal parabolic subgroup containing $B$ associated to $\alpha$.

\begin{defn}
The {\em Bott-Samelson variety} $Z_{\mfi}$ corresponding to a word $\mfi=(i_1, \dots, i_n) \in \{1,2,\ldots,r\}^n$ is the quotient
\[Z_{\mfi}:=(P_{\beta_1}\times\cdots\times
P_{\beta_n})/B^n\]
where $\be_j = \alpha_{i_j}$ and $B^n$ acts on the right on $P_{\beta_1} \times \cdots \times
P_{\beta_n}$ by:
\[(p_1,\dots,p_n)\cdot
(b_1,\dots,b_n):=(p_1b_1,b_1^{-1}p_2b_2,\dots,b_{n-1}^{-1}p_nb_n).\]

\end{defn}
It is known that $Z_{\mfi}$ is a smooth projective algebraic variety of dimension
$n$. 
By convention, if $n=0$ and $\mfi$ is the empty word, we set
$Z_{\mfi}$ equal to a point. There is a natural morphism from $Z_{\mfi}$ to the flag variety $G/B$:
\begin{equation}\label{eq: def mu}
\mu:Z_{\mfi}\rightarrow G/B, [p_1,\dots,p_n]\mapsto p_1\dots p_nB.
\end{equation}
When the word $\mfi$ is reduced (that is, $s_{\be_1}\dots s_{\be_n}$ is a reduced word decomposition of an element in $W$), this morphism is birational to its image, which is the Schubert variety corresponding to  $s_{\be_1}\dots s_{\be_n}$.
In particular, when $s_{\be_1}\dots s_{\be_n}$ is the longest element  in the Weyl group, $\mu$ is a birational morphism from $Z_{\mfi}$ to $G/B$. \emph{For the remainder of this note, we assume that the word $\mfi$ is reduced.}

Recall that, in Step 1 of the construction of Newton-Okounkov bodies described above, we choose an embedding of the given projective variety $X$ into a projective space.  This choice corresponds to  a choice of a very ample line bundle   on the variety.  We now review known facts about line bundles on Bott-Samelson varieties. Let $Z_{\mfi}$ be a Bott-Samelson variety associated to a reduced word $\mfi = (i_1, \ldots, i_n)$ as above. Suppose given a sequence $\{\la_1,\dots,\la_n\}$ of weights
$\la_j \in \Lambda$. We let $\C^*_{(\la_1,\dots,\la_n)}$ denote the one-dimensional representation of $B^n$ defined by
\begin{equation}\label{def:rep}
(b_1,\dots,b_n)\cdot k:=e^{-\la_1}(b_1)\cdots e^{-\la_n}(b_n)k.
\end{equation}
We define
the line bundle $L_{\mfi}(\lambda_1, \ldots, \lambda_n)$ over
$Z_{\mfi}$ by 
\begin{equation}\label{eq:definition line bundle}
L_{\mfi}(\lambda_1, \ldots, \lambda_n):=(P_{\be_1}\times\cdots
P_{\be_n})\times_{B^n} \C^*_{(\la_1,\dots,\la_n)}
\end{equation}
where the equivalence relation is given by
 \[((p_1, \ldots, p_n) \cdot (b_1, \ldots, b_n),k)\sim ((p_1, \ldots,
 p_n), (b_1, \ldots, b_n) \cdot k)\]
 for $(p_1, \ldots, p_n) \in P_{\beta_1} \times\cdots\times
 P_{\beta_n}, (b_1, \ldots, b_n) \in
 B^n$, and $k\in\C$. The projection $L_{\mfi}(\lambda_1, \ldots, \lambda_n) \to Z_{\mfi}$
  is given by taking the first factor
 $[(p_1, \ldots, p_n, k)] \mapsto [(p_1, \ldots, p_n)] \in Z_{\mfi}$.

In what follows, we will frequently choose the weights $\lambda_j$ to be of a special form, as follows.
Specifically, suppose given a {\em multiplicity list}
$\mfm=(m_1,\dots,m_n) \in \Z_{\geq 0}^n$. Then we may define a
sequence of weights $\{\la_1, \ldots, \la_n\}$ associated to the
word $\mfi$ and the multiplicity list $\mfm$ by setting
\begin{equation}\label{def:lambda}
\la_1:=m_1\varpi_{\be_1},\dots, \la_n=m_n\varpi_{\be_n}.
\end{equation}
In this special case we will use the notation
\begin{equation}\label{eq:definition Lim}
L_{\mfi,\mfm} := L_{\mfi}(m_1 \varpi_{\beta_{1}},
\cdots, m_n \varpi_{\beta_{n}}).
\end{equation}

\begin{prop}[\protect\cite{Lauritzen-Thomsen}] \label{prop:line bundle}
Let $\mfi, Z_{\mfi}, \mfm, L_{\mfi,\mfm}$ be as above. The line bundle $L_{\mfi,\mfm}$ over $Z_{\mfi}$ is globally generated (respectively very ample) if $m_i\ge 0 (\text{respectively } m_i>0)$ for all $i=1,\dots, n$.
\end{prop}

The Borel subgroup acts on both $Z_{\mfi}$ and $L_{\mfi,\mfm}$ by left multiplication on the first factor.
Thus the space of global sections $H^0(Z_{\mfi},L_{\mfi,\mfm})$ is
naturally a $B$-module; these are called \textit{generalized Demazure
  modules} (cf. for instance \cite{LLM02}).

We now recall two constructions of
Newton-Okounkov bodies of Bott-Samelson varieties from
\cite{HY2,Fujita} which are, a priori, different.  We will later see that, under certain conditions, they actually coincide
(Proposition \ref{prop:equivalence}). This fact allows us to prove the main result of this note in Section~\ref{sec:main}.

Let $\mfi$ and $\mfm$ be as above. Let $R=R_{\mfi,\mfm}=\oplus_{d\ge 0} H^0(Z_{\mfi}, L_{\mfi,\mfm})^d$, the homogeneous coordinate ring of the image of $Z_{\mfi}$ under the Kodaira morphism corresponding to $L_{\mfi,\mfm}$. 
Part of the construction of a Newton-Okounkov body of $R$ is a choice of a valuation on $R$. In both \cite{HY2} and \cite{Fujita} the valuation is constructed using a system of local coordinates $(t_1,\dots, t_n)$ on $Z_{\mfi}$ near the point $[e,e,\dots,e]\in Z_{\mfi}$, where $e$ is the identity element of the group $G$. Specifically, the coordinate system is defined by the map 
\begin{equation}\label{eq:coords}
\C^n\hookrightarrow Z_{\mfi}, \quad (t_1,\dots,t_n)\mapsto [\exp(t_1F_{\be_1}),\dots,\exp(t_nF_{\be_n})]
\end{equation} 
where $F_{\be_i}$ is the Chevalley generator for the root subspace $\mkg_{-\be_i} (i=1,\dots, n)$.  Denote by $\mc{U}$ the image of $\C^n$ under this embedding.  Consider the following two valuations on the polynomial ring $\C[t_1,\dots,t_n]$:
\begin{itemize}
  \item $\nu_{\max}$: the highest-term valuation with respect to the lexicographic order with $t_1>\dots>t_n$ defined as follows. For a nonzero-polynomial $f\in \C[t_1,\dots,t_n]$, let $t_1^{a_1}\cdots t_n^{a_n}$ be the highest term of $f$. Then we define \[\nu_{\max}(f):= (-a_1,\dots,-a_n).\]
  \item $\nu_{\min^{op}}$: the lowest-term valuation with respect to the lexicographic order with $t_n>\dots>t_1$, defined as follows. For a nonzero-polynomial $f\in \C[t_1,\dots,t_n]$, let $t_1^{b_1}\cdots t_n^{b_n}$ be the lowest term of $f$. Then we define \[\nu_{\min^{op}}(f):= (b_1,\dots,b_n).\]
\end{itemize}

To define a valuation on sections of a line bundle, we need a choice of a normalization, i.e., a base section which is nowhere vanishing on the local coordinate system $\mc{U}$. We now construct our choice of base section. 
For a simple root $\beta$, let $\hat{P}_{\beta}$ be the maximal parabolic subgroup corresponding to $\beta$. It is well-known that $H^0(G/\hat{P}_{\beta}, \mathcal{L}_{\varpi_{\beta}}) \cong V^*_{\varpi_{\beta}}$ as $G$-representations, where 
\begin{equation}\label{eq: L lambda}
\mathcal{L}_{\varpi_{\beta}} := (G \times_{\hat{P}_{\beta}} \C_{\varpi_{\beta}} \to G/\hat{P}_{\beta})
\end{equation}
is the $G$-equivariant line bundle associated to the weight $\varpi_{\beta}$ and $V_{\varpi_{\beta}}$ is the irreducible $G$-representation of highest weight $\varpi_{\beta}$. Moreover, up to non-zero scalars, there is a unique non-zero lowest-weight vector in any irreducible $G$-representation. These representations are related to the Bott-Samelson variety $Z_{\mfi}$ through the following closed immersion of $Z_{\mfi}$ into a product of Grassmann varieties: 
\[\phi: Z_{\mfi}\rightarrow G/\hat{P}_{\be_1}\times\cdots \times G/\hat{P}_{\be_n}, [p_1,\dots,p_n]\mapsto ([p_1],[p_1p_2],\dots, [p_1p_2\cdots p_n])\]
and the projection 
\[
\pi_i:G/\hat{P}_{\be_1}\times\cdots \times G/\hat{P}_{\be_n}\rightarrow G/\hat{P}_{\be_i}
\]
to the $i$-th factor of this product. It is known that the line bundle $L_{\mfi, \mfm}$ defined above can also be described in terms of the morphisms $\phi$ and $\pi_i$, namely, 
\[
L_{\mfi, \mfm} \cong \phi^* ( \otimes_{i=1}^n \pi^* (\mathcal{L}_{m_i \varpi_{\beta_i}})) \cong \phi^* (\otimes_{i=1}^n 
\pi^* (\mathcal{L}_{\varpi_{\beta_i}}^{\otimes m_i})) \cong \otimes_{i=1}^n \left( \phi^* \pi^*(\mathcal{L}_{\varpi_{\beta_i}}^{\otimes m_i}) \right).
\]
For the remainder of this discussion, and for each simple root $\beta_i$, we fix a choice 
\[
0 \neq \tau_i  \in H^0(G/\hat{P}_{\beta_i} , \mathcal{L}_{\varpi_{\beta_i}}) \cong V^*_{\varpi_{\beta_i}}
\]
of non-zero lowest weight vector (equivalently, $B^-$-eigenvector).  This choice is unique up to scalars, and the properties of the sections $\tau_i$ which are required for later arguments are independent of this choice. In what follows, by slight abuse of notation, we denote by $\tau_i^{\otimes m_i}$ the element of $H^0(G/\hat{P}_{\beta_i}, \mathcal{L}_{m_i \varpi_{\beta_i}})$ which is the image of $\tau_i^{\otimes m_i} \in H^0(G/\hat{P}_{\beta_i}, \mathcal{L}_{\varpi_{\beta_i}})^{\otimes m_i}$ under the natural ($G$-equivariant) map 
\[
H^0(G/\hat{P}_{\beta_i}, \mathcal{L}_{\varpi_{\beta_i}})^{\otimes m_i} \to H^0(G/\hat{P}_{\beta_i}, \mathcal{L}_{m_i \varpi_{\beta_i}}).
\]
Note that by standard representation theory, $\tau_i^{\otimes m_i}$ is a lowest-weight vector of 
$H^0(G/\hat{P}_{\beta_i}, \mathcal{L}_{m_i \varpi_{\beta_i}}) \cong V^*_{m_i \varpi_{\beta_i}}$. Given a reduced word $\mfi = (i_1,\ldots,i_n)$ and multiplicity list $\mfm=(m_1,\ldots,m_n) \in \Z^n_{\geq 0}$ we may now define a section 
$\tau_{\mathbf{i}, \mathbf{m}}$ by 
\begin{equation}\label{eq: def tau}
\tau_{\mathbf{i}, \mathbf{m}} := \otimes_{i=1}^n \left( \phi^* \pi^* (\tau_i^{\otimes m_i} ) \right)
\end{equation}
where, once again by abuse of notation, we have identified the RHS with its image under the map 
\[
\otimes_{i=1}^n H^0(Z_{\mathbf{i}}, \phi^* \pi^* \mathcal{L}_{\varpi_{\beta_i}}^{\otimes m_i} ) 
\to H^0(Z_{\mathbf{i}}, \otimes_{i=1}^n \left( \phi^* \pi^*(\mathcal{L}_{\varpi_{\beta_i}}^{\otimes m_i}) \right) ).
\]
The following is immediate.

\begin{lem}\label{lem: tau is additive}
Let $\mathbf{m}, \mathbf{m'}, \mathbf{m}'' \in \Z^n_{\geq 0}$ and suppose $\mathbf{m} = \mathbf{m}' + \mathbf{m}''$. Then $L_{\mfi, \mfm} \cong L_{\mfi, \mfm'} \otimes L_{\mfi, \mfm''}$ and $\tau_{\mathbf{i}, \mathbf{m}} = \tau_{\mathbf{i}, \mathbf{m}'} \otimes \tau_{\mathbf{i}, \mathbf{m}''}$. \end{lem} 

The following is also known (see for example \cite[Section 2.3]{Fujita}). 

\begin{lem}\label{lem: tau nonvanishing}
The section $\tau_{\mfi, \mfm}$ never vanishes on the coordinate neighborhood $\mc{U} \cong \C^n$ of~\eqref{eq:coords}. 
\end{lem} 

Now, given any $0 \neq \sigma \in H^0(Z_{\mfi}, L_{\mfi, \mfm}) = V$ we may define the valuation 
$v_{min^{op}}$ (respectively $v_{max}$) on $V$ by $v_{min^{op}}(\sigma) := \nu_{min^{op}}(\frac{\sigma}{\tau_{\mfi, \mfm}})$ (respectively $v_{max}(\sigma) := \nu_{max}(\frac{\sigma}{\tau_{\mfi, \mfm}})$). This also naturally extends to $R_k$ for $k>1$ by sending $\sigma \in H^0(Z_{\mfi}, L_{\mfi, \mfm}^{\otimes k}) \mapsto \nu_{min^{op}}(\frac{\sigma}{\tau_{\mfi,\mfm}^{\otimes k}})$ (and similarly for $v_{max}$).

The Newton-Okounkov bodies $\Delta(Z_{\mfi}, L_{\mfi, \mfm}, v_{max})$ of Bott-Samelson varieties $Z_{\mfi}$ with respect to the line bundles $L_{\mfi, \mfm}$ and the valuation 
$v_{max}$ are computed in \cite{Fujita} and are denoted $-\Delta_{\mfi,\mfm}$. On the other hand, under a technical hypothesis, the Newton-Okounkov bodies $\Delta(Z_{\mfi}, L_{\mfi,\mfm}, v_{{\min}^{op}})$ with respect to $v_{\min^{op}}$ are found in \cite{HY2} and are denoted by $P_{\mfi,\mfm}^{op}$. In both cases, these Newton-Okounkov bodies are convex polytopes, defined by a concrete set of inequalities, as we now describe.

We begin with the polytope $P_{\mfi,\mfm}$. (The polytope $P_{\mfi,\mfm}^{op}$ is obtained from $P_{\mfi,\mfm}$ by reversing the order of the coordinates $(x_1,x_2,\ldots,x_n) \mapsto (x_n, \ldots, x_2, x_1)$.) 
Let $\mfi=(i_1,\ldots,i_n)$ be a reduced word and $\mfm=(m_1,\ldots,m_n) \in \Z^n_{\geq 0}$ a multiplicity list as above. We define a set of functions $A_j=A_j(x_1,\ldots,x_n)$ associated to $\mfi,\mfm$ as follows (the functions depend on $\mfi, \mfm$ but for simplicity we suppress it from the notation): 
\begin{equation}\label{def: As}
\begin{array}{l}
A_n:= m_n,\\
A_{n-1}(x_n):= \lee m_{n-1}\varpi_{\be_{n-1}}+m_n\varpi_{\be_n}-x_n\be_n,\be_{n-1}^{\vee}\ree,\\
A_{n-2}(x_{n-1},x_n):=\lee m_{n-2}\varpi_{\be_{n-2}}+m_{n-1}\varpi_{\be_{n-1}}+m_{n}\varpi_{\be_n}-x_{n-1}\be_{n-1}-x_n\be_n,\be_{n-2}^{\vee}\ree\\
 \quad \vdots\\
A_1(x_2,\dots,x_n):=\lee m_1\varpi_{\be_1}+m_2\varpi_{\be_2}+\cdots + m_n\varpi_{\be_n}-x_2\be_2-\cdots-x_n\be_n,\be_1^{\vee}\ree
\end{array}
\end{equation}
Note the function $A_j$ in fact depends only on the variables $x_n, x_{n-1}, \ldots, x_{j+1}$, and the notation used above reflects this. We can now define the polytopes in question. 

\begin{defn}\label{def:polytope}
 The polytope $P_{\mfi,\mfm}$ is the
set of all points $\mf{x}=(x_1,\dots,x_n) \in \R^n$ satisfying  the following inequalities:
\[0\le x_j\le A_j(x_{j+1},\dots,x_n), \quad1\le j\le n.\]
\end{defn}

\begin{rem}\label{remark:toric}
 The polytope $P_{\mfi,\mfm}$ has appeared previously in the
  literature and has connections to toric geometry and representation theory. Specifically, under a hypothesis on $\mfi$ and $\mfm$
  which we call ``condition \textbf{(P)}'' (see Definition~\ref{definition:conditionP} below), we show
  in \cite{HY1} that $P_{\mfi,\mfm}$ is exactly a so-called
  \textit{Grossberg-Karshon twisted cube}. These twisted cubes were introduced in
  \cite{Grossberg-Karshon1994} in connection with Bott towers and
  character formulae for irreducible $G$-representations. The arguments in \cite{HY1} use a certain torus-invariant
  divisor in a toric variety associated to Bott-Samelson varieties
  studied by Pasquier \cite{Pasquier}.
\end{rem}

Now we introduce the condition {\bf (P)} which is a
 technical hypothesis on the word and the multiplicity
list.

\begin{defn}\label{definition:conditionP}
We say that the pair $(\mfi, \mfm)$
  \textbf{satisfies condition (P)} if
  \begin{enumerate}
  \item[(P-n)] $m_n \geq 0$
  \end{enumerate}
  and for every integer $k$ with $1 \leq k \leq n-1$, the following statement, which we refer to as condition (P-k),
    holds:
\begin{enumerate}
\item[(P-k)]
      if $(x_{k+1}, \ldots, x_n)$ satisfies
     \begin{equation*}
      \begin{array}{ccccl}
      0& \leq& x_n& \leq& A_n \\
      0& \leq &x_{n-1}& \leq& A_{n-1}(x_n) \\
    && \vdots&& \\
     0 &\leq& x_{k+1}& \leq& A_{k+1}(x_{k+2}, \ldots, x_n), \\
       \end{array}
       \end{equation*}
   then \[A_k(x_{k+1},\ldots, x_n) \geq 0.\] \end{enumerate}
In particular, condition \textbf{(P)} holds if and only if the conditions (P-1) through (P-n) all hold.
\end{defn}

\begin{eg}\label{eg:1} Let $G=\SL_3$,  $\mfi=(1,2,1)$ and $\mfm=(0,1,1)$. Then for $x_3=A_3=1$ and $x_2=0\le A_2(x_3)$, $A_1(x_2,x_3)=m_1+m_3+x_2-2x_3=-1<0$. Thus $(\mfi,\mfm)$ does not satisfy the condition {\bf (P)}. See also Example \ref{eg:2}.
\end{eg}

\begin{rem}
  The condition \textbf{(P)} is rather restrictive. On the other hand,
  for a given word $\mfi$, it is not
  difficult to explicitly construct multiplicity lists $\mfm$ such that $(\mfi,\mfm)$ satisfying  condition
  \textbf{(P)}, as described in Proposition \ref{prop: nice m}.
\end{rem}

The following theorem shows that the polytopes $P_{\mfi,\mfm}$ (up to reversal of coordinates) are Newton-Okounkov bodies of Bott-Samelson varieties.

\begin{thm}[\protect{\cite[Theorem 3.4]{HY2} \label{thm:HY}}]

Suppose that $(\mfi,\mfm)$ satisfies the condition {\bf (P)}. Then the semigroup $S(Z_{\mfi},L_{\mfi,\mfm},v_{\min^{op}})$ is finitely generated and the corresponding Newton-Okounkov body is equal to \[P_{\mfi,\mfm}^{op}:=\{(x_1,\dots,x_n)\in \R^n| (x_n,\dots,x_1)\in P_{\mfi,\mfm}\}.\]

\end{thm}

We now describe the polytope $\De_{\mfi,\mfm}$ which appears in \cite{Fujita} and which is called the {\em generalized string polytope} \cite[\S2.3]{Fujita}. Moreover, we compare  $\De_{\mfi,\mfm}$ with the polytope $P_{\mfi,\mfm}$ described above and show that, under the condition (P), these polytopes coincide. It turns out that this is a simple consequence of the general fact that the volume of a Newton-Okounkov body (for a fixed line bundle) is independent of the choice of valuation.

Given an $\mf{x}=(x_1,\dots, x_n)\in\Z_{\ge 0}^n$, we inductively define a sequence $\mf{x}^{(n)},\dots,\mf{x}^{(1)}$ with $\mf{x}^{(k)}\in\Z^k\quad (k=1,\dots,n)$ as follows. We define 
$\mf{x}^{(n)}=\mf{x},$ and then define $\mf{x}^{(k-1)}=(x_1^{(k-1)},\dots,x_{k-1}^{(k-1)})$ by the rule 
\[
x_i^{(k-1)}:=\min\{x_i^{(k)}, \Psi^{(k)}(i)\}\quad (i=1,\dots, k-1)
\] 
where $\Psi^{(k)}:\{1,\dots, k-1\}\rightarrow\Z$ is the function defined as follows. 
For $i<j\le k$ we define  
\[s(i,j,k):=x^{(k)}_{j}-\sum_{i<s\le j }  \lee\be_s, \be_k^{\vee}\ree x_s^{(k)}+\sum_{i\le s<j} \de_{i_k,i_s}m_s\] 
and then define 
\[\Psi^{(k)}(i):=\left\{\begin{array}{lc}\max\{s(i,j,k) \hspace{2mm} | \hspace{2mm} i<j\le k, \be_j=\be_k\}, & \text{if }\be_i=\be_k; \\
x_i^{(k)},&\text{ otherwise.}
\end{array}\right.\]
We have the following. 

\begin{defn}\label{def: string polytope}
The polytope $\De_{\mfi,\mfm}$  is the set of all points $\mf{x}=(x_1,\dots,x_n) \in \R^n$  satisfying the following inequalities: 
\[\Psi^{(k)}(i)\ge 0, \quad 2\le k\le n,\quad 1\le i\le k-1;\]
\[0\le x_j\le A_j(x_{j+1},\dots,x_n), \quad1\le j\le n.\]
\end{defn} 

\begin{rem}\label{rem: string}
For $\mfi$ a reduced word and $\mfm=\mfm(\la)$ for a dominant weight $\lambda$, then the above polytope is the usual string polytope associated to the Weyl group element $w$ corresponding to $\mfi$ and the dominant weight $\la$ \cite[Remark 5.11]{Fujita}. 
\end{rem} 

\begin{thm}\label{thm:Fujita}\cite[Theorem 3.1, Corollary 3.2]{Fujita} The semigroup $S(Z_{\mfi}, L_{\mfi,\mfm}, \nu_{\max})$ is finitely generated and
the corresponding Newton-Okounkov body is equal to
\[-\De_{\mfi,\mfm}:=\{(x_1,\dots,x_n)\in \R^n: (-x_1,\dots,-x_n)\in \De_{\mfi,\mfm}\}.\]
\end{thm}

It is immediate from the definitions that the polytope
$\De_{\mfi,\mfm}$ is a subset of the polytope $P_{\mfi,\mfm}$. In fact, under the condition {\bf (P)}, more is true, as we record in the following proposition. 

  \begin{prop} \label{prop:equivalence} 
  Suppose $(\mfi, \mfm)$ satisfies the condition {\bf (P)}. Then the  polytope $\De_{\mfi,\mfm}$ is equal to the polytope $P_{\mfi, \mfm}$.
\end{prop}

\begin{proof}
From Theorem \ref{thm:HY} and Theorem~\ref{thm:Fujita} we know that both 
 $P_{\mfi,\mfm}^{op}$ and $-\De_{\mfi, \mfm}$ are Newton-Okounkov bodies of $Z_{\mfi}$ with respect to the same line bundle $L_{\mfi,\mfm}$, but with respect to different valuations. The volume of a Newton-Okounkov body is independent of the choice of valuation, so these convex polytopes have the same volume. Since reversal of coordinates and taking negatives both preserve volume, we conclude $\De_{\mfi, \mfm}$ and $P_{\mfi, \mfm}$ have the same volume. On the other hand, we have already seen that $\De_{\mfi, \mfm} \subseteq P_{\mfi, \mfm}$. Hence they must be equal. 

\end{proof}

\begin{rem} 
The description of the polytope $P_{\mfi,\mfm}$ given in Definition~\ref{def:polytope} is quite explicit and the functions $A_j$ are relatively simple. Thus one can see, for example, that if $(\mfi,\mfm)$ satisfies the condition {\bf (P)}, then the polytope $P_{\mfi,\mfm}$ is a lattice polytope whose vertices are easily computed \cite{HY1}.
Thus, from Proposition~\ref{prop:equivalence} above, under the condition {\bf (P)}, we can deduce that 
$\De_{\mfi,\mfm}$  is also a lattice polytope. (Fujita proves that the $\De_{\mfi,\mfm}$ are rational polytopes in \cite{Fujita}.) 
Moreover, under the condition {\bf (P)}, Proposition~\ref{prop:equivalence} implies that the inequalities $\Psi^{(k)}(i)\ge 0$ appearing in the definition of the $\De_{\mfi,\mfm}$ are redundant. Since the functions $\Psi^{(k)}(i)$ are rather more complicated than the functions $A_j$, these facts seem not so straightforward to prove directly. 

\end{rem} 
 
\section{A functorial desingularization of singularities}\label{sec:main}

As mentioned in the introduction, toric degenerations of flag varieties are actively studied due to their connections to e.g. representation theory and Schubert calculus. For example, Kiritchenko, Smirnov, and Timorin give a description of the multiplicative structure of the cohomology ring of the type A flag variety in terms of certain subsets of faces of the Gel'fand-Zeitlin polytope in \cite{KST}.  A complication that arises in these considerations is that the Gel'fand-Zeitlin polytopes -- and more generally, string polytopes-- are singular (i.e. its associated toric variety is singular). Thus we were led to wonder whether one could associate to a flag variety $G/B$ and a dominant weight $\la \in \Lambda$ a choice of non-singular toric variety which fits naturally in a diagram relating it to both $G/B$ and the toric variety of its string polytope $\De_\la$. Our approach to this question is to employ a Bott-Samelson variety $Z_{\mfi}$ which is a birational model of $G/B$. Our main result, Theorem~\ref{thm:main}, is one potential answer to this question, and its main assertion is the existence of a commutative diagram involving two toric degenerations which relate all four varieties in question: the flag variety $G/B$, the toric variety of $\De_\la$, the Bott-Samelson variety $Z_\mfi$, and a non-singular toric variety corresponding to certain auxiliary data. In particular, our non-singular toric variety provides a resolution of singularities of the toric variety of $\De_\la$.

Before stating our main theorem, we recall some standard results which relate the geometry of Bott-Samelson varieties and flag varieties. Let $\lambda \in \Lambda$ be a dominant weight and let $\mathcal{L}_\lambda = (G \times_B \C_{\lambda} \to G/B)$ be the associated $G$-equivariant line bundle over $G/B$ as in~\eqref{eq: L lambda}. Using the map $\mu$ defined in~\eqref{eq: def mu} we can pull back $\mathcal{L}_{\lambda}$ to a line bundle on $Z_{\mfi}$. The following is well-known \cite[Chapters 13 and 14]{Jantzen}. 

\begin{lem}\label{lem:1} Let $\mfi=(i_1, \ldots,i_n)$ be a reduced word for $w=w_0$ the longest element in the Weyl group. Then $H^0(Z_{\mfi}, \mu^*\mathcal{L}_{\la}) \cong H^0(G/B, \mathcal{L}_{\la})$.  
\end{lem}

It is also known that there exists a special choice $\mfm(\lambda) \in \Z^n_{\geq 0}$ of multiplicity list such that the pullback bundle $\mu^*(\mathcal{L}_{\la})$ is isomorphic to $L_{\mfi, \mfm(\lambda)}$.  
Specifically, for a dominant weight $\la \in \Lambda$ we define $\mfm(\la)$ as follows. 
Write $\la=\ell_1\varpi_1+\cdots+\ell_r\varpi_r$ with respect to the fundamental weights $\varpi_i (i=1,\dots, r)$.  Suppose the rightmost occurrence of $i=1$ in the word $\mfi$ is at position $k$: that is, $i_k=1, i_j\ne 1$ for $ j>k$. If this is the case, we define $\mfm(\la)_k=\ell_1$. (If $i=1$ does not occur in $\mfi$, then proceed to the next step.) Next, suppose $k'$ be the rightmost occurrence of $i=2$ in $\mfi$. Then we define $\mfm(\la)_{k'}=\ell_2$. Proceed in this way for each $i=1,\dots, r$. Finally, we define $\mfm(\la)_j=0$ if it has not already been defined in the previous steps.
We have the following \cite{Jantzen}. 

\begin{lem}\label{lem:2}
Let $\mfi=(i_1,\ldots,i_n)$ be a reduced word and $\la \in \Lambda$ be a dominant weight. Then the pull-back bundle $\mu^*(\mathcal{L}_{\la})$ over $Z_{\mfi}$ is isomorphic to $L_{\mfm(\la)}$, where $\mfm(\la)$ is defined as above. 
\end{lem}

Finally, note that the results of Anderson as recounted in Section~\ref{sec:background} imply that there exists a flat family 
$\mk{X}'$ such that the generic fiber of $\mk{X}'$ is isomorphic to $G/B$ and the special fiber of $\mk{X}'$ is $\Proj \C[\Gamma']$ where $\Gamma' = \Gamma(R', \nu_{max})$. Since $\Gamma'$ is the set of lattice points contained in a strongly convex rational polyhedral cone \cite[Section 3.2 and Theorem 3.10]{BZ} it can be seen from standard results on toric varieties \cite[Section 1.2 and Section 2.1]{CLS} that $\Proj \C[\Gamma']$ is normal. Thus the singular fiber is isomorphic to the toric variety $\mr{Tor}(\De_{\mfi,\mfm(\la)})$ corresponding to the string polytope $\De_{\mfi, \mfm(\la)}$ (cf. Remark~\ref{rem: string}). As noted already, $\mr{Tor}(\De_{\mfi,\mfm(\la)})$ may in general be singular.  The point of our main result, which we can now state, is to provide a functorial desingularization of the toric variety $\mr{Tor}(\De_{\mfi,\mfm(\la)})$ which fits in a larger commutative diagram involving a (smooth) Bott tower and the family $\mk{X}'$ mentioned above. 

\begin{thm}\label{thm:main}
Let $\mfi=(i_1,\ldots,i_n)$ be a reduced word for $w_0$ the longest element of $W$. Let $\lambda \in \Lambda$ a dominant weight and $\mfm(\lambda) \in \Z^n_{\geq 0}$ the corresponding multiplicity list, as defined above. Let $\mk{X}'$ be the family described above. 
Then there exists a multiplicity list $\mfm=(m_1,\ldots,m_n) \in \Z^n_{\geq 0}$ and 
a one-parameter flat family $\mk{X}$ with the following properties: 

\begin{enumerate}
\item The polytope $\De_{\mfi,\mfm}$ is a smooth lattice polytope and thus $\mr{Tor}(\De_{\mfi,\mfm})$ is smooth. 
\item The generic fiber of $\mk{X}$ is isomorphic to $Z_{\mfi}$.
\item The special fiber of $\mk{X}$ is isomorphic to the smooth toric variety $\mr{Tor}(\De_{\mfi,\mfm})$ of the (generalized) string polytope $\De_{\mfi, \mfm}$.
\item We have $\Delta_{\mfi, \mfm(\la)} \subseteq \Delta_{\mfi, \mfm}$ and the birational map $\mu: Z_{\mfi} \to G/B$ of~\eqref{eq: def mu} degenerates to a birational morphism of toric varieties $\psi_0: \mr{Tor}(\De_{\mfi, \mfm}) \to \mr{Tor}(\De_{\mfi, \mfm(\la)})$, that is, there is a commutative diagram of families
 \[\xymatrix{\mk{X}\ar[rr]^{\Psi}
\ar[dr] && \mk{X}'\ar[dl]\\
&\C&
}\] such that $\Psi_t \cong \mu$ for $t \neq 0$ and $\Psi_0 = \psi_0$.

\end{enumerate}

\end{thm}

\begin{rem}
The smooth toric variety $\mr{Tor}(\De_{\mfi,\mfm})$ which appearing above is an instance of a \textbf{Bott tower} (see e.g. \cite{Grossberg-Karshon1994}). 
\end{rem}

To prove the above theorem, it suffices to check that we can find a multiplicity list $\mfm=(m_1,\ldots,m_n)$ in such a way that the corresponding data satisfy the hypotheses of Anderson's Theorem~\ref{thm:functoriality}.  The following proposition is the main technical result of this section, from which Theorem~\ref{thm:main} easily follows. 
Moreover, our proof of Proposition~\ref{prop: nice m} shows that there are many choices of multiplicity lists $\mfm$ for which Theorem~\ref{thm:main} holds.

\begin{prop}\label{prop: nice m} 
Let $\mfi=(i_1,\ldots,i_n)$ be a reduced word. Let $\lambda \in \Lambda$ be a dominant weight and let $\mfm(\la) \in \Z^n_{\geq 0}$ be the multiplicity list associated to $\la$ as defined above. 
Then there exists a multiplicity list $\mfm=(m_1,\dots,m_n) \in \Z^n_{\geq 0}$ such that
\begin{enumerate}
\item $\mfm(\la)_i \le m_i  \textup{ for all } i=1,\dots, n,$
\item$ m_i>0 \textup{ for all }  i=1,\dots, n $, or equivalently, $L_{\mfi,\mfm}$ is very ample, 
\item the pair $(\mfi, \mfm)$ satisfies the condition {\bf (P)}, and 
\item the polytope $P_{\mfi,\mfm}$ is a smooth lattice polytope.
\end{enumerate}

\end{prop}

Before embarking on the proof of Proposition~\ref{prop: nice m} we briefly review some 
results related to the polytope $P_{\mfi,\mfm}$. In \cite{Pasquier}, Pasquier describes a degeneration of $Z=Z_{\mfi}$ to a smooth toric variety $X(\Si_{\mfi})$; the fan $\Si_{\mfi}$ in $\R^n$ is complete and smooth, and $X(\Si_{\mfi})$ is commonly called a Bott tower. The polytope $P_{\mfi,\mfm}$ is in fact the polytope of a torus-invariant divisor $D_{\mfi,\mfm}$ on $X(\Si_{\mfi})$ \cite{HY1}.  The Cartier data of the divisor $D_{\mfi, \mfm}$ can be described explicitly, as follows \cite[Lemma 1.9]{HY1}. Firstly, the  maximal cones of $\Si_{\mfi}$  are in 1-1 correspondence with the set $\{+,-\}^n$. Secondly, for each maximal cone $\si=(\si_1,\dots,\si_n)\in\{+,-\}^n$, the associated Cartier data $\mk{r}_{\si}=(\mk{r}_{\si,1},\dots,\mk{r}_{\si,n})\in\Z^n$ is given by the formula
\begin{equation}\label{eq: def Cartier}
\mk{r}_{\si,i}=\left\{\begin{array}{lc} 0 & \text{ if }\si_i=+\\ A_i(\mk{r}_{\si,i+1},\dots, \mk{r}_{\si,n})& \text{ if }\si_i=-
\end{array}\right. 
\end{equation}
for all $i=1,\dots, n$, where the $A_i$ are the functions defined in~\eqref{def: As}.

\begin{proof}[Proof of Proposition \ref{prop: nice m} ] 
From \cite[Proposition 2.1, Theorem 2.4]{HY1} we know that the following $3$ statements are equivalent: 
(i) the pair $(\mfi, \mfm)$ satisfies the condition {\bf (P)}, and (ii) 
the divisor $D_{\mfi,\mfm}$ described above is basepoint-free, and (iii) 
$ \mk{r}_{\si,i}\ge 0$ for all $\si\in\{+,-\}^n $and all $i=1,\dots, n$. 
Also, it is a basic fact from toric geometry  \cite[\S 6.1]{CLS} that $D_{\mfi,\mfm}$ is basepoint-free if and only if $\{\mk{r}_{\si}|\si\in\{+,-\}^n\}$ is the set of vertices of $P_{\mfi,\mfm}$. Further, the divisor $D_{\mfi,\mfm}$ is very ample if, in addition, the elements $\mk{r}_{\si} (\si\in\{+,-\}^n)$ are pairwise distinct.
Finally, if $\{ \mk{r}_{\si} \hspace{2mm} \vert \hspace{2mm} \si\in\{+,-\}^n \}$ is equal to the set of vertices of $P_{\mfi, \mfm}$ and its elements are all distinct, then the normal fan of $P_{\mfi,\mfm}$ is equal to $\Si_{\mfi}$ \cite[Theorem 6.2.1]{CLS}. By the results of \cite{Pasquier} as above, we know $\Si_{\mfi}$ is smooth, so we may conclude $P_{\mfi,\mfm}$ is a smooth polytope. Moreover, if $(\mfi, \mfm)$ satisfies the condition {\bf (P)}, then from \cite{HY1} we know $P_{\mfi, \mfm}$ is a lattice polytope. 
It follows that, in order to prove the proposition, it suffices to find  
$\mfm=(m_1,\dots,m_n) \in \Z^n$ such that 
\begin{enumerate} 
\item[(M1)] $\mfm(\la)_i \le m_i $, for all $i=1, \ldots,n$, 
\item[(M2)] $m_i > 0$, for all $i=1, \ldots, n$, 
\item[(M3)] $\mk{r}_{\si,i}\ge  0$ for all $i=1,\dots,n$ and for all $\si \in \{+,-\}^n$, and 
\item[(M4)] the $\mk{r}_{\si}$ are distinct for all $\si \in \{+,-\}^n$. 
\end{enumerate} 
From the description of the vectors $\mk{r}_{\si}$ given in~\eqref{eq: def Cartier} 
it also follows that the conditions (M3) and (M4) above would be satisfied if $\mfm=(m_1,\ldots,m_n)$ has the property that 
\begin{equation}\label{eq: third condition}
\mk{r}_{\si, i} = A_i(\mk{r}_{\si, i+1},\dots, \mk{r}_{\si, n})> 0 \textup{  if $\si_i = -$, for any $\si \in \{+,-\}^n$ and for all $i, 1 \leq i \leq n$.}
\end{equation}
(Recall the functions $A_i$ depend on $\mfi$ and $\mfm$ although this is suppressed from the notation.) 

It now suffices to find an $\mfm=(m_1,\ldots,m_n)$ satisfying conditions (M1), (M2), and~\eqref{eq: third condition}. We will construct each $m_i$ inductively, starting with $i=n$.
Let $m_n$ be any integer greater than or equal to $\max\{\mfm(\la)_n,1\}$. Then, by construction, conditions (M1) and (M2) are satisfied for $i=n$; moreover, $\mk{r}_{\si,n}=A_n :=m_n>0$ for any $\si\in\{+,-\}^n$ with $\si_n=-$, so~\eqref{eq: third condition} is also satisfied for $i=n$. 
Now suppose that $k<n$ and also suppose by induction that the integers 
  $m_{k+1},\dots, m_n$ have been chosen to satisfy the conditions (M1), (M2), and~\eqref{eq: third condition} for all $i=\ell$ where $k+1 \leq \ell \leq n$. 
It is straightforward from the definition of the functions $A_i$ and the formula~\eqref{eq: def Cartier} 
that the value of $\mk{r}_{\sigma, \ell}$ for any $\sigma \in \{+,-\}^n$ is uniquely determined by the constants $m_{\ell}, m_{\ell+1}, \ldots, m_n$. Therefore, under the inductive hypothesis that the constants $m_{k+1}, \ldots, m_n$ have already been defined, we can set 
\[
M_{\ell}:=\max\{\mk{r}_{\si,\ell}|\si\in\{+,-\}^n\} 
\]
for all $\ell$ satisfying $k+1 \leq \ell \leq n$. Now we may choose 
$m_k$ to be any integer greater than or equal to 
\[
\max \left\{ \mfm(\la)_{k} \hspace{1mm},  \quad 1 \hspace{1mm}, \quad 1- \left(\sum_{\substack{k+1 \leq \ell \leq n \\ \be_\ell=\be_k}}(m_{\ell}-2M_{\ell}) \right) \right\}.
\] 
From the choice of $m_k$ it immediately follows that conditions (M1) and (M2) are satisfied for $i=k$. We now wish to show that~\eqref{eq: third condition} is also satisfied for $i=k$. 
Suppose $\si\in\{+,-\}^n$ with $\si_k=-$. We have 
\[\begin{array}{ccccl} \mk{r}_{\si,k}&=&A_k(\mk{r}_{\si,k+1},\dots, \mk{r}_{\si,n})&=&m_k+\sum\limits_{\ell>k,\be_k=\be_{\ell}}m_{\ell}-\sum\limits_{\ell>k} \lee \be_{\ell}, \be_k^{\vee}\ree \mk{r}_{\si,\ell}\\ &&&\ge&

m_k+\sum\limits_{\ell>k,\be_k=\be_{\ell}}(m_{\ell}- 2 \mk{r}_{\si,\ell})\\
 &&&\ge&

m_k+\sum\limits_{\ell>k,\be_k=\be_{\ell}}(m_{\ell}- 2 M_{\ell})\\
 &&&>&

0

\end{array}\]
where the first equality is by~\eqref{eq: def Cartier}, the second equality is the definition of $A_k$, 
the first inequality is because $\mk{r}_{\al,\ell}\geq 0$ for $\ell > k$ by the inductive hypothesis and $\lee \be_{\ell},\be_k^{\vee}\ree\le 0$ if $\be_{\ell}\ne\be_k$, the second inequality is by the definition of $M_{\ell}$, and the last inequality follows from the choice of $m_k$. Thus, we conclude that $\mk{r}_{\sigma, k} > 0$ for any $\sigma$ with $\sigma_k = -$, as was to be shown. Continuing inductively, it follows that we may choose constants $m_1, m_2, \ldots, m_n$ such that conditions (M1), (M2), and~\eqref{eq: third condition} are satisfied for all $i=1, \ldots, n$. This completes the proof.

\end{proof}

As advertised, the proof of Theorem~\ref{thm:main} follows straightforwardly from the above proposition.

\begin{proof}[Proof of Theorem \ref{thm:main}.]

Choose $\mfm=(m_1,\ldots,m_n)$ satisfying the conditions given in Proposition~\ref{prop: nice m}, whose existence is guaranteed by that proposition. Since the pair $(\mfi, \mfm)$ satisfies the condition {\bf (P)} by construction, Proposition~\ref{prop:equivalence} implies that $\Delta_{\mfi, \mfm} = P_{\mfi, \mfm}$. From Condition (4) of Proposition~\ref{prop: nice m} we may then conclude that $\Delta_{\mfi, \mfm}$ is a smooth lattice polytope, and thus that $\mr{Tor}(\Delta_{\mfi,\mfm})$ is a smooth toric variety, as desired. This proves the condition (1) in the statement of the theorem.

Consider the line bundle $L_{\mfi, \mfm}$ over $Z_{\mfi}$ associated to the choice of multiplicity list $\mfm$ in the previous paragraph. Note that $L_{\mfi, \mfm}$ is very ample by Proposition~\ref{prop:line bundle}, since by construction we have $m_i > 0$ for all $i$. 
Let $V = H^0(Z_{\mfi}, L_{\mfi, \mfm})$ and, following the notation of Section~\ref{sec:background}, 
let $R=R(V)$ denote the homogeneous coordinate ring associated to the Kodaira embedding of $Z_{\mfi}$ with respect to $L_{\mfi, \mfm}$. Let $v_{max}$ denote the valuation on $R$ used in Theorem~\ref{thm:Fujita}, which has one-dimensional leaves. By Theorem~\ref{thm:Fujita} we know that the corresponding semigroup $\Gamma=S(Z_{\mfi}, L_{\mfi, \mfm}, \nu_{max})$ is finitely generated, so we can apply Theorem~\ref{thm:NOBY degeneration} to obtain a flat family $\mk{X} \to \mathbb{A}^1$. Moreover, by Theorem~\ref{thm:NOBY degeneration} the flat family $\mk{X}$ has generic fiber isomorphic to $\Proj R$ and special fiber isomorphic to $\Proj \C[\Gamma]$. But by our construction, $\Proj R \cong Z_{\mfi}$. Thus the condition (2) in the statement of the theorem is satisfied for this choice of $\mfm$ and family $\mk{X}$.  

Next we claim that $\Proj \C[\Gamma]$ is normal. To see this, first note that from the definitions~\eqref{def: As} of the functions $A_k$ and of the condition {\bf (P)} it is straightforward that if $(\mfi, \mfm)$ satisfies condition {\bf (P)} then so does $(\mfi, r\mfm)$ for any positive integer $r$. Using this fact together with Proposition~\ref{prop:equivalence}, \cite[Proposition 2.5]{HY2} and an argument analogous to \cite[Proof of Theorem 3.4]{HY2}, which in particular depends on the fact that $P_{\mfi,\mfm}$ is a lattice polytope \cite[Theorem 2.4]{HY1}, it follows that $\Gamma$ is generated in level $1$. It then follows that $\Proj \C[\Gamma]$ is normal (see e.g. \cite[Lemma 3 and Remark 1]{Anderson}) and is isomorphic to $\mr{Tor}(\De_{\mfi, \mfm})$.
Thus, for this family $\mk{X} \to \mathbb{A}^1$, the condition (3) in the statement of the theorem is also satisfied. 

It remains to prove the property (4) in the statement of the theorem. To see this, it suffices to check that the rings $R(V)$ and $R(V')$ for $V' = H^0(Z_{\mfi}, L_{\mfi, \mfm(\la)}) \cong H^0(G/B, \mathcal{L}_{\la})$ satisfy the compatibility hypotheses of Theorem~\ref{thm:functoriality}. 
Specifically, we must prove that there exists an inclusion map $V' \hookrightarrow V$ such that the valuation $\nu_{max}$ on $R'$ restricts to the valuation $\nu_{max}$ on $R$, and in addition, we must prove that the corresponding rational map $\psi:
  X=\Proj (R)\rightarrow X'=\Proj(R')$ is a birational
  isomorphism.  We begin with the first claim. 
  By the construction in Proposition~\ref{prop: nice m} we have that $m_i \geq \mfm(\la)_i$ for all $i=1,\ldots,n$. Define $\mfm' := (m'_1, \ldots, m'_n)$ where $m'_i := m_i - \mfm(\la)_i$ for all $i$. Then $L_{\mfi, \mfm} \cong L_{\mfi, \mfm(\la)} \otimes L_{\mfi, \mfm'}$ by Lemma~\ref{lem: tau is additive}. Also let $\tau_{\mfi, \mfm'}$ be the section of $H^0(Z_{\mfi}, L_{\mfi, \mfm'})$ defined in~\eqref{eq: def tau}. This is a non-zero section by Lemma~\ref{lem: tau nonvanishing}. 
We can now define a map 
\[
\iota: V'=H^0(Z_{\mfi}, L_{\mfi, \mfm(\la)}) \to V = H^0(Z_{\mfi}, L_{\mfi, \mfm}), \quad \sigma \mapsto \sigma \otimes \tau_{\mfi, \mfm'}
\]
where by slight abuse of notation we use $\sigma \otimes \tau_{\mfi, \mfm'}$ to denote its image under the natural map $H^0(Z_{\mfi}, L_{\mfi, \mfm(\la)}) \otimes H^0(Z_{\mfi}, L_{\mfi, \mfm'}) \to H^0(Z_{\mfi}, L_{\mfi, \mfm(\la)} \otimes L_{\mfi, \mfm'} =  L_{\mfi, \mfm})$. Since $Z_{\mfi}$ is irreducible and $\tau_{\mfi, \mfm'}$ is non-vanishing on an open dense subset, $\iota$ is injective. 
We extend this inclusion to $R(V')=R'\hookrightarrow R(V)=R$ by sending $s\in V'^{k}$ to $s\otimes \tau_{\mfi, \mfm'}^{\otimes k}$.
Moreover, the definition of the valuation $\nu_{max}$ on $R$ (respectively $R'$) normalizes holomorphic sections using $\tau_{\mfi, \mfm}$ (respectively $\tau_{\mfi, \mfm(\la)}$). Thus, since $\tau_{\mfi, \mfm} = \tau_{\mfi, \mfm(\la)} \otimes \tau_{\mfi, \mfm'}$ by Lemma~\ref{lem: tau is additive} we conclude that, given $\sigma \in V'$, we have 
$ v_{\max}(\sigma/\tau_{\mfi,\mfm(\la)})=v_{\max}(\sigma\otimes \tau_{\mfi,\mfm'}/\tau_{\mfi,\mfm(\la)}\otimes \tau_{\mfi,\mfm'})=v_{\max}(\iota(\sigma)/\tau_{\mfi,\mfm})$ as desired. The argument extends straightforwardly to higher degree components. This shows the first claim.

 Finally, we need to show that the map $\psi: \Proj(R) \to \Proj(R')$ is a birational isomorphism. 
 The corresponding morphism of function fields $\Psi: K(X')\cong R'_{((0))}\rightarrow K(X) \cong R_{((0))}$ is given by $\frac{s_1}{s_2}\mapsto \frac{s_1\otimes \tau_{i,\mfm'}^k}{s_2\otimes \tau_{i,\mfm'}^k}$, where $k$ is the common degree of $s_1$ and $s_2$ \cite[Theorem I.3.4(c), page 18]{Hartshorne}. 
 However, $\tau_{i,\mfm'}$ never vanishes on an open dense subset of $Z_{\mfi}$. Thus, as a rational functions, $\frac{s_1\otimes \tau_{i,\mfm''}^k}{s_2\otimes \tau_{i,\mfm''}^k}$ is equal to $\frac{s_1}{s_2}$. Hence $\Psi$ is injective. We also know a priori that $K(X')\cong K(X)$, since $\mu: \Proj(R') \cong G/B \to \Proj(R) \cong Z_{\mfi}$ is known to be a birational isomorphism. Thus $\Psi$ must be an isomorphism.  
Now $X$ and $X'$ are integral schemes of finite
  types over $\C$ and so $\psi$ must be a birational isomorphism \cite[Section 6.5]{Vakil}.

\end{proof}

We illustrate our main result with an example. 

\begin{eg}\label{eg:2}
Let $G=\SL_3, \mfi=(1,2,1)$, and  $\la=\al_1+\al_2$. Then $\mfm(\la)=(m_1',m_2',m_3')=(0,1,1)$.  In Example \ref{eg:1} it is shown that $(\mfi,\mfm(\la))$ does not satisfy the condition {\bf (P)}. As the figure shows below, $P_{\mfi,\mfm(\la)}$ is not a lattice polytope.
Following the proof of the Lemma above, we choose $m_3$ to be any integer $\ge m_3'=1$.  For $m_2$, we choose any integer $\ge m_2'=1$.
For $m_1$, we choose any integer $\ge\max\{m_1', 1-m_3+2M_3=1+m_3\}$.
 Note that if we take $\mfm=(1,1,1)$, then $(\mfi,\mfm)$ still satisfies the condition {\bf(P)}. However, the corresponding polytope $P_{\mfi,\mfm}$ is not a simple polytope.

\begin{picture}
  (100,100)(-70,-25)
 \put(0,0){\vector(1,0){50}}\put
(0,0){\vector(0,1){70}}\put
(0,0){\vector(-1,-1){30}}
\linethickness{0.25mm}
\put(-1, 53){$x_1$}\put(51,0){$x_2$}
\put(-37,-35){$x_3$}
\put(-3,-3){$\bullet$}\put(17,-3){$\bullet$}\put(-3,17){$\bullet$}\put(17,17){$\circ$}\put(17,37){$\bullet$}
{\color{red}\put(-10,-10){$\bullet$}}\put(0,-20){$\bullet$}\put(25,-20){$\bullet$}\put(25,0){$\bullet$}
\put(-30,-15){\small{$(0,0,\frac{1}{2})$}}
\thicklines
\drawline(0,0)(-7,-7)\drawline(-7,-7)(2,-18)\drawline(2,-18)(27,-18)\drawline(27,-18)(27,2)\drawline(27,2)(2,-18)\drawline(2,-18)(0,20)
\drawline(0,20)(-7,-7)\drawline(0,20)(20,40)\drawline(20,40)(20,0)\drawline(20,40)(27,2)\drawline(20,0)(27,-18)
\put(-25,-40){$\mfm=(0,1,1)$}
\end{picture}

\begin{picture}
(0,0)(-210,-35)

 \put(0,0){\vector(1,0){60}}\put
(0,0){\vector(0,1){70}}\put
(0,0){\vector(-1,-1){30}}
\linethickness{0.25mm}\put(-3,-3){$\bullet$}\put(-3,37){$\bullet$}\put(17,-3){$\bullet$}\put(17,57){$\bullet$}
\put(-20,-20){$\bullet$}\put(25,-20){$\bullet$}\put(25,20){$\bullet$}
\put(-3,-3){$\circ$}\put(17,-3){$\circ$}\put(-3,17){$\circ$}\put(17,17){$\circ$}\put(17,37){$\circ$}
\put(0,-20){$\circ$}\put(25,-20){$\circ$}\put(25,0){$\circ$} \put(0,-2){$\circ$}
\thicklines
\drawline(-18,-18)(27,-18)\drawline(27,-18)(27,22)\drawline(27,22)(20,60)
\drawline(20,60)(20,0)\drawline(20,60)(0,40)\drawline(0,40)(-18,-18)
\drawline(-18,-18)(27,22)\drawline(20,0)(27,-18)\drawline(0,0)(0,37)\drawline(-2,0)(18,0)
\drawline(0,0)(-18,-18)
\put(-25,-40){$\mfm=(1,1,1)$}
\end{picture}

\begin{picture}
   (0,0)(-350,-50)
 \put(-4,0){\vector(1,0){70}}\put
(-4,0){\vector(0,1){70}}\put
(-4,0){\vector(-1,-1){30}}

\linethickness{0.25mm}\put(-6,-3){$\bullet$}\put(-6,17){$\circ$}\put(-6,37){$\circ$}
\put(-6,59){$\bullet$}\put(-25,0){$\bullet$}\put(-24,-20){$\bullet$}\put(22,-20){$\bullet$}
\put(22,38){$\bullet$}

\put(15,77){$\bullet$}

\put(0,0){$\circ$}\put(15,-3){$\bullet$}\put(0,20){$\circ$}\put(15,17){$\circ$}\put(15,37){$\circ$}\put(15,57){$\circ$}
\put(0,-20){$\circ$}\put(22,0){$\circ$} \put(22,20){$\circ$}
\thicklines
\drawline(-5,60)(-23,3)\drawline(-23,3)(-23,-18)\drawline(-23,-18)(24,-18)\drawline(24,-18)(24,40)
\drawline(24,40)(17,80)\drawline(17,80)(-4,61)
\put(15,77){$\bullet$}\drawline(24,40)(-23,3)\drawline(-4,-2)(-4,60)\drawline(-6,-2)(-22,-18)
\drawline(-5,0)(16,0)\drawline(24,-18)(17,0)\drawline(18,78)(17,0)

\put(-25,-42){$\mfm=(2,1,1)$}

\end{picture}

\end{eg}



\end{document}